\documentclass[11pt]{article}

\usepackage[utf8]{inputenc}
\usepackage[T1]{fontenc}
\usepackage{lmodern}

\usepackage{amsmath, amssymb, amsthm}
\usepackage{mathtools}

\usepackage{graphicx}
\usepackage{hyperref}

\usepackage{enumitem}
\usepackage{microtype}

\graphicspath{{figures/}}
\hypersetup{hidelinks}

\theoremstyle{plain}
\newtheorem{theorem}{Theorem}[section]
\newtheorem{lemma}[theorem]{Lemma}

\newtheorem{corollary}[theorem]{Corollary}
\newtheorem{conjecture}[theorem]{Conjecture}

\theoremstyle{definition}
\newtheorem{definition}[theorem]{Definition}
\newtheorem{question}[theorem]{Question}
\newtheorem{problem}[theorem]{Problem}

\theoremstyle{remark}
\newtheorem{remark}[theorem]{Remark}

\title{Order 14 is the largest order for which every 4-total coloring of every cubic graph is equitable}

\author{
Matheus Adauto\thanks{Instituto de Computação, Universidade Federal do Rio de Janeiro, Rio de Janeiro, Brazil. Corresponding author: \texttt{adauto@ic.ufrj.br}}
\and
Celina de Figueiredo\thanks{Programa de Engenharia de Sistemas e Computação, Universidade Federal do Rio de Janeiro, Rio de Janeiro, Brazil. \texttt{celina@cos.ufrj.br}}
\and
Diana Sasaki\thanks{Instituto de Matemática e Estatística, Universidade do Estado do Rio de Janeiro, Rio de Janeiro, Brazil. \texttt{diana.sasaki@ime.uerj.br}}
\and
Rafael Schneider\thanks{Programa de Engenharia de Sistemas e Computação, Universidade Federal do Rio de Janeiro, Rio de Janeiro, Brazil. \texttt{rschneider@cos.ufrj.br}}
}
\date{}

\begin{document}

\maketitle

\begin{abstract}
A total coloring of a graph is an assignment of colors to its vertices
and edges so that adjacent or incident elements receive distinct
colors, and it is equitable when the cardinalities of any two color
classes differ by at most one. Stemock conjectured that every $4$-total
coloring of a cubic graph of order less than $20$ is equitable. In this
paper, we disprove this conjecture: the circular ladder $L_{12}$ admits
a non-equitable $4$-total coloring and, moreover, no smaller
counterexample exists: order $4$ is vacuous, and every $4$-total
coloring of a cubic graph of order $6$, $8$, or $10$ is equitable. We
also prove that the same property holds at order $14$. Our proofs
rely on a decomposition lemma, which states that, in any $4$-total
coloring of a cubic graph $G$, each color class consists of an
independent set $S$ together with a perfect matching of $G-S$. We use
the lemma to determine all possible color class configurations for
orders $12$, $16$, and $18$, and we show that every listed
configuration is attained. Finally, we provide a splicing construction
showing that, for every even $n\geq16$, some connected cubic graph of
order $n$ admits a non-equitable $4$-total coloring. We may conclude
that $14$ is the largest order for which every $4$-total coloring of
every cubic graph is equitable.
\end{abstract}

\medskip

\section{Introduction}\label{sec:intro}
 
Total coloring is a classical topic in graph theory that involves
assigning colors to the vertices and edges of a graph, subject to
certain rules. A \emph{$k$-total coloring} of a graph $G$ assigns one
of $k$ colors to each vertex and each edge of $G$ in such a way that
adjacent vertices, adjacent edges, and any vertex and edge that are
incident receive distinct colors. The least $k$ for which $G$ admits a
$k$-total coloring is the \emph{total chromatic number} $\chi''(G)$.
Clearly $\chi''(G) \geq \Delta(G)+1$, where $\Delta(G)$ denotes the
maximum degree of $G$. The celebrated Total Coloring Conjecture, posed
independently by Behzad~\cite{behzad} and Vizing~\cite{vizing}, states
that $\chi''(G) \leq \Delta(G)+2$ for every simple graph $G$;
see~\cite{yap} for a thorough account. Graphs with
$\chi''(G) = \Delta(G)+1$ are said to be \emph{Type~1}, and graphs with
$\chi''(G) = \Delta(G)+2$ are \emph{Type~2}. The conjecture holds for
cubic graphs, since one can check that every graph with maximum degree
$3$ satisfies $\chi'' \leq 5$ (see~\cite{feng} for a concise proof),
and so every cubic graph is either Type~1 ($\chi''=4$) or Type~2
($\chi''=5$).
 
A $k$-total coloring is \emph{equitable} if the cardinalities of any
two of its color classes differ by at most one, and the \emph{equitable
total chromatic number} $\chi''_e(G)$ is the least $k$ for which $G$
admits an equitable $k$-total coloring. Wang~\cite{wang} proved that
every graph with maximum degree $3$ satisfies $\chi''_e \leq 5$, which
confirms, for cubic graphs, the natural equitable analogue of the Total
Coloring Conjecture. It was shown in~\cite{cubic} that deciding whether
a bipartite cubic graph admits an equitable $4$-total coloring is
NP-complete. The same work determined the equitable total chromatic
number of ladder-like families and exhibited a remarkable cubic graph
$R$ of order $20$. The graph $R$ is Type~1 and satisfies
$\chi''_e(R) = 5$; that is, \emph{no} $4$-total coloring of $R$ is
equitable; see Figure~\ref{fig:R}.
 
\begin{figure}[htbp]
\centering
\includegraphics[width=.60\linewidth]{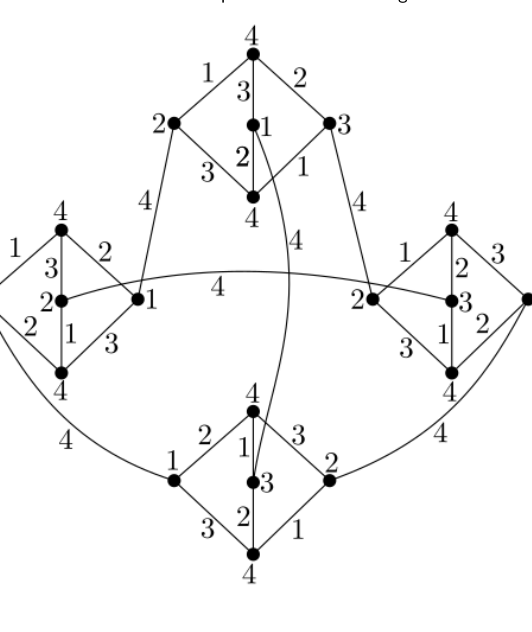}
\caption{A non-equitable $4$-total coloring of the graph $R$
from~\cite{cubic}. Vertex and edge labels indicate their colors.}
\label{fig:R}
\end{figure}
 
Motivated by the graph $R$, Stemock~\cite{bryson} investigated
equitable total colorings of regular graphs and posed the following
conjecture.
 
\begin{conjecture}[Stemock~\cite{bryson}]\label{conj:stemock}
Every $4$-total coloring of a cubic graph of order less than $20$ is
equitable.
\end{conjecture}
 
Notice that the scope of Conjecture~\ref{conj:stemock} is considerable,
since there are exactly $45{,}982$ connected cubic graphs of order less
than~$20$~\cite{bcgm}.
 
In this paper, we settle Conjecture~\ref{conj:stemock} and, more
generally, we determine exactly for which orders the conjectured
phenomenon occurs. Until now, it was known that the property fails at
order $20$, because of the graph $R$. First, we disprove the conjecture
by exhibiting non-equitable $4$-total colorings of the circular ladders
$L_{12}$ and $L_{18}$ (Section~\ref{sec:counterexamples}). Second, we
prove that every $4$-total coloring of a cubic graph of order $6$, $8$,
$10$, or $14$ is equitable (Section~\ref{sec:small}), and so $L_{12}$
is a counterexample of smallest possible order. Third, we establish a
structural decomposition of the $4$-total colorings of cubic graphs
(Section~\ref{sec:structure}): each color class is an independent set
$S$ together with a perfect matching of $G-S$. The decomposition yields
the complete lists of possible color class configurations for orders
$12$, $16$, and $18$, and we display explicit examples showing that
every listed configuration is attained. Finally, in
Section~\ref{sec:larger} we introduce a splicing operation that merges
two totally colored cubic graphs into a connected one while adding the
cardinalities of their color classes. Hence, for every even
$n \geq 16$, some connected cubic graph of order $n$ admits a
non-equitable $4$-total coloring. We may conclude that $n=14$ is the
largest order for which every $4$-total coloring of every cubic graph
is equitable.

We present the necessary definitions and notation in
Section~\ref{sec:prelim}, and we prove our results in
Sections~\ref{sec:structure} to~\ref{sec:larger}. We conclude in
Section~\ref{sec:conclusion} with additional comments and questions.
Preliminary versions of the results of Sections~\ref{sec:small}
and~\ref{sec:counterexamples} appeared in the first author's M.Sc.
dissertation~\cite{adauto}.
 
\section{Preliminaries}\label{sec:prelim}
 
In this paper, every graph $G$ is simple, finite, and cubic
($3$-regular) unless stated otherwise. The numbers of vertices and
edges of $G$ are denoted by $n$ and $m$, respectively; for a cubic graph, $n$ is even, $m = 3n/2$, and
the total number of \emph{elements} (vertices and edges) is $5n/2$. For
$S \subseteq V(G)$, we write $G - S$ for the subgraph induced by
$V(G) \setminus S$, and for a matching $M$ we write $V(M)$ for the set
of endpoints of the edges of $M$.
 
Given a $4$-total coloring $\mathcal{C}$ of $G$ with color set
$\{1,2,3,4\}$, the \emph{color class} $c_i$ is the set of elements of
color $i$. Each color class is a \emph{total independent set}: a set of
elements that are pairwise neither adjacent nor incident. The coloring
$\mathcal{C}$ is \emph{equitable} if
$\bigl||c_i|-|c_j|\bigr| \leq 1$ for all $i,j$; otherwise it is
\emph{non-equitable}. We call the multiset
$\{|c_1|,|c_2|,|c_3|,|c_4|\}$ the \emph{color class configuration} of
$\mathcal{C}$, written as a non-increasing sequence.
 
We follow the notation of~\cite{cubic, bryson} for the graph
families involved. The \emph{circular ladder} $L_n$ ($n$ even,
$n \geq 6$) is the prism $C_{n/2} \,\square\, K_2$; the edges joining
the two copies of $C_{n/2}$ are its \emph{rungs}. More generally,
$G(p,k)$ denotes the generalized Petersen graph on $2p$ vertices, so
that $L_{2p} = G(p,1)$ and the Petersen graph is $G(5,2)$. The
M\"obius ladder of order $n$ is denoted $M_n$. Among cubic graphs, the
complete graph $K_4$, the complete bipartite graph $K_{3,3}$, and the
M\"obius ladders are Type~2, and circular ladders are
Type~1~, except for $L_{10}$, which is Type~2 \cite{chet, yap}.
 
\section{The structure of \texorpdfstring{$4$}{4}-total colorings of cubic graphs}
\label{sec:structure}
 
The simple yet powerful lemma below is the basis of all of our
results. It states that a $4$-total coloring of a cubic graph is the
same thing as a partition of the vertex set into four independent sets
together with a compatible partition of the edge set into four
matchings.
 
\begin{lemma}[Decomposition lemma]\label{lem:decomposition}
Let $G$ be a cubic graph of order $n$ and let $\mathcal{C}$ be a
$4$-total coloring of $G$ with color set $\{1,2,3,4\}$. For each
$i \in \{1,2,3,4\}$, let $S_i$ be the set of vertices of color $i$ and
let $M_i$ be the set of edges of color $i$. Then:
\begin{enumerate}
\item[(i)] $(S_1,S_2,S_3,S_4)$ is a partition of $V(G)$ into (possibly
empty) independent sets;
\item[(ii)] for each $i$, the set $M_i$ is a perfect matching of
$G - S_i$;
\item[(iii)] $(M_1,M_2,M_3,M_4)$ is a partition of $E(G)$;
\item[(iv)] for each $i$, the cardinality $|S_i|$ is even and
$|c_i| = \bigl(n + |S_i|\bigr)/2$.
\end{enumerate}
Conversely, if $(S_1,S_2,S_3,S_4)$ is a partition of $V(G)$ into
independent sets and $(M_1,M_2,M_3,M_4)$ is a partition of $E(G)$ such
that each $M_i$ is a perfect matching of $G - S_i$, then assigning
color $i$ to the vertices of $S_i$ and to the edges of $M_i$ yields a
$4$-total coloring of $G$.
\end{lemma}
 
\begin{proof}
Statements (i) and (iii) are immediate: every element receives exactly
one color, and adjacent vertices receive distinct colors, so each $S_i$
is independent.
 
(ii) The set $M_i$ is a matching, since adjacent edges receive distinct
colors, and no edge of $M_i$ is incident to a vertex of $S_i$, since a
vertex and an incident edge receive distinct colors; hence
$M_i \subseteq E(G - S_i)$. Now let $u \notin S_i$. The four elements
consisting of $u$ and its three incident edges are pairwise adjacent or
incident, so they receive four distinct colors; since only four colors
are available, every color occurs among them. Since the color of $u$ is
not $i$, some edge incident to $u$ has color $i$, that is, $M_i$
saturates $u$. Therefore $M_i$ is a perfect matching of $G - S_i$.
 
(iv) By (ii), $|M_i| = \bigl(n - |S_i|\bigr)/2$, so $n - |S_i|$ is
even; since $n$ is even, $|S_i|$ is even. Moreover
$|c_i| = |S_i| + |M_i| = \bigl(n + |S_i|\bigr)/2$.
 
For the converse, notice that every element receives exactly one color.
Adjacent vertices lie in distinct sets $S_i$ because each $S_i$ is
independent; adjacent edges receive distinct colors because each $M_i$
is a matching; and no vertex of $S_i$ is incident to an edge of $M_i$
because $M_i \subseteq E(G - S_i)$. Hence all three constraints of a
total coloring are satisfied.
\end{proof}
 
Throughout the paper we write $v_i = |S_i|$. Whenever we refer to the
\emph{vertex profile} of a coloring, we relabel the colors if necessary
so that $v_1 \geq v_2 \geq v_3 \geq v_4$; the resulting non-increasing
sequence $(v_1,v_2,v_3,v_4)$ is the vertex profile. By
Lemma~\ref{lem:decomposition}(iv), the vertex profile determines the
color class configuration, and summing over the four classes gives
$v_1+v_2+v_3+v_4 = n$, in agreement with (i). The following criterion for equitability is immediate; a related
formulation appears in~\cite{diana}.
 
\begin{corollary}\label{cor:equitable-criterion}
A $4$-total coloring of a cubic graph is equitable if and only if
$|v_i - v_j| \leq 2$ for all $i,j$, that is, if and only if
$v_1 - v_4 \leq 2$.
\end{corollary}
 
\begin{proof}
By Lemma~\ref{lem:decomposition}(iv),
$\bigl||c_i|-|c_j|\bigr|=|v_i-v_j|/2$ for all $i,j$.
\end{proof}
 
We remark that Lemma~\ref{lem:decomposition}(iv) recovers, for cubic
graphs, the classical parity constraint of conformable vertex
colorings~\cite{chet} without invoking conformability: the parity of
each vertex class is forced by the existence of the perfect matchings
$M_i$.
 
The next lemma collects the two counting bounds that drive all case
analyses in Sections~\ref{sec:small} and~\ref{sec:counterexamples}.
 
\begin{lemma}[Counting bounds]\label{lem:bounds}
Let $\mathcal{C}$ be a $4$-total coloring of a cubic graph $G$ of order
$n$, with vertex classes $S_1,\dots,S_4$ and $v_i = |S_i|$. Then:
\begin{enumerate}
\item[(a)] $5 v_i \leq 2n$ for each $i$;
\item[(b)] $v_i + 2 v_j \leq n$ for all $i \neq j$.
\end{enumerate}
\end{lemma}
 
\begin{proof}
(a) Since $S_i$ is independent, the edges incident to vertices of $S_i$
number exactly $3 v_i$, and none of them belongs to $M_i$, because
$M_i$ avoids $S_i$ by Lemma~\ref{lem:decomposition}(ii). These
$3 v_i$ edges together with the $(n - v_i)/2$ edges of $M_i$ are
pairwise distinct. Hence
$3 v_i + (n - v_i)/2 \leq m = 3n/2$; that is, $5 v_i \leq 2n$.
 
(b) Since the vertex classes partition $V(G)$, we have
$S_j \subseteq V(G) \setminus S_i = V(M_i)$. The two endpoints of any
edge of $M_i$ are adjacent, so the independent set $S_j$ contains at
most one endpoint of each edge of $M_i$. Hence
$v_j \leq |M_i| = (n - v_i)/2$.
\end{proof}
 
\begin{definition}\label{def:feasible}
A non-increasing sequence $(v_1,v_2,v_3,v_4)$ of non-negative integers
is a \emph{feasible profile} for order $n$ if all its entries are even,
$v_1+v_2+v_3+v_4 = n$, $5v_1 \leq 2n$, and $v_2 + 2v_1 \leq n$.
\end{definition}
 
By Lemmas~\ref{lem:decomposition} and~\ref{lem:bounds}, the vertex
profile of any $4$-total coloring of a cubic graph of order $n$ is a
feasible profile for $n$ (notice that, for a non-increasing sequence,
condition (b) of Lemma~\ref{lem:bounds} holds for all $i \neq j$ as
soon as $v_2 + 2v_1 \leq n$). Feasibility is a necessary condition
only: whether a feasible profile is actually attained is a question
about the particular graph, addressed in
Sections~\ref{sec:counterexamples} and~\ref{sec:larger}.
 
\section{Orders \texorpdfstring{$6$, $8$, $10$, and $14$}{6, 8, 10, and 14}}\label{sec:small}
 
Notice that the smallest cubic graph $K_4$ is Type~2~\cite{yap}, and so
order $4$ is trivial, since $K_4$ admits no $4$-total coloring. We now
handle the first non-trivial orders in a uniform way, through the
decomposition lemma.
 
\begin{theorem}\label{thm:small}
Let $G$ be a cubic graph of order $n \in \{6, 8, 10, 14\}$. Then every
$4$-total coloring of $G$ is equitable.
\end{theorem}
 
\begin{proof}
Let $(v_1,v_2,v_3,v_4)$ be the vertex profile of a $4$-total coloring
of $G$; it is a feasible profile for $n$. By
Corollary~\ref{cor:equitable-criterion}, it suffices to show that every
feasible profile for $n$ satisfies $v_1 - v_4 \leq 2$.
 
If $n = 6$, then $5v_1 \leq 12$ gives $v_1 \leq 2$, and the only
feasible profile is $(2,2,2,0)$.
 
If $n = 8$, then $5v_1 \leq 16$ gives $v_1 \leq 3$, hence $v_1 \leq 2$
by parity, and the only feasible profile is $(2,2,2,2)$.
 
If $n = 10$, then $5v_1 \leq 20$ gives $v_1 \leq 4$. If $v_1 = 4$, then
$v_2 + 2v_1 \leq 10$ gives $v_2 \leq 3$, hence $v_2 \leq 2$, and the
sum condition forces $(4,2,2,2)$. If $v_1 \leq 2$, the sum of the four
entries is at most $8 < 10$, a contradiction. Hence the only feasible
profile is $(4,2,2,2)$.
 
If $n = 14$, then $5v_1 \leq 28$ gives $v_1 \leq 5$, hence
$v_1 \leq 4$. Four even entries at most $4$ summing to $14$ force the
profile $(4,4,4,2)$.
 
In every case we have $v_1 - v_4 \leq 2$, and so the coloring is
equitable, which completes the proof.
\end{proof}
 
The corresponding color class configurations are $(4,4,4,3)$ for
$n = 6$, $(5,5,5,5)$ for $n = 8$, $(7,6,6,6)$ for $n = 10$, and
$(9,9,9,8)$ for $n = 14$. In particular, Theorem~\ref{thm:small}
generalizes the known fact that every $4$-total coloring of the
Petersen graph is equitable~\cite{petersen}.

\section{Counterexamples and color class configurations}
\label{sec:counterexamples}
 
Order $12$ is the first order that admits a feasible profile with
$v_1-v_4\geq4$, and indeed Conjecture~\ref{conj:stemock} already fails
there. The counterexamples are the circular ladders $L_{12}$ and
$L_{18}$; see Figure~\ref{fig:ladders} for explicit
non-equitable $4$-total colorings; their color class configurations are
$(8,8,8,6)$ and $(12,12,12,9)$, respectively. One can check that these
configurations also follow from the circular-ladder construction of
Chetwynd and Hilton~\cite{chet}.

\begin{figure}[htbp]
\centering
\includegraphics[width=\linewidth]{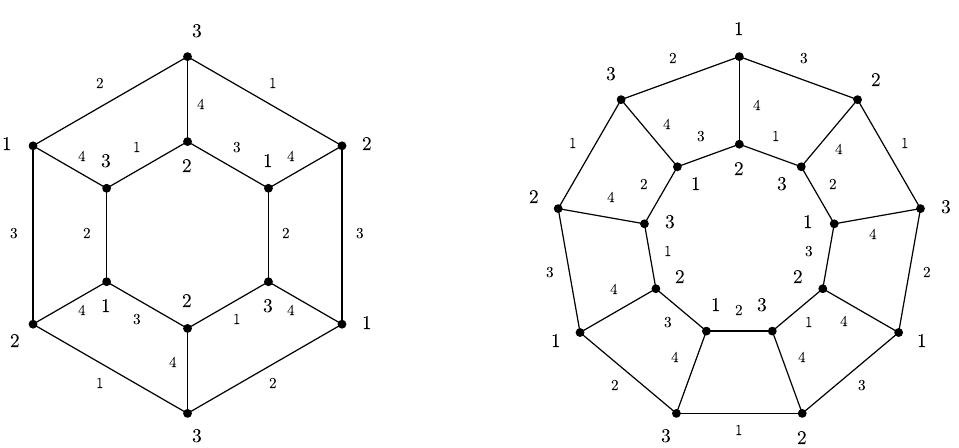}
\caption{Non-equitable $4$-total colorings of $L_{12}$ (left) and $L_{18}$ (right). Vertices are drawn in black, and the labels on vertices and edges indicate their colors.}
\label{fig:ladders}
\end{figure}
 
\begin{corollary}\label{cor:smallest}
Conjecture~\ref{conj:stemock} is false, and $L_{12}$ is a
counterexample of smallest possible order: every cubic graph of order
less than $12$ has all of its $4$-total colorings equitable.
\end{corollary}
 
\begin{proof}
The coloring of $L_{12}$ in Figure~\ref{fig:ladders} is a
$4$-total coloring with color class configuration $(8,8,8,6)$, which is
not equitable. For orders less than $12$, notice that order $4$ is
vacuous, since $K_4$ is Type~2, and that orders $6$, $8$, and $10$ are
covered by Theorem~\ref{thm:small}. This concludes the proof.
\end{proof}
 
We remark that both $L_{12}$ and $L_{18}$ also admit \emph{equitable}
$4$-total colorings; indeed, $\chi''_e(L_n) = 4$ for every circular
ladder~\cite{cubic}.

We now determine, for each of the orders $12$, $16$, and $18$, all
possible color class configurations and show that every possibility is
attained.
 
\begin{theorem}\label{thm:n12}
In every $4$-total coloring of a cubic graph of order $12$, the color
class configuration is $(8,8,7,7)$ or $(8,8,8,6)$. Both configurations
are attained by $L_{12}$; in particular, $(8,8,8,6)$ is the unique
configuration of a non-equitable coloring.
\end{theorem}
 
\begin{proof}
Bound (a) of Lemma~\ref{lem:bounds} gives $v_1 \leq 24/5$, hence
$v_1 \leq 4$. The multisets of four even numbers at most $4$ summing to
$12$ are $(4,4,4,0)$ and $(4,4,2,2)$; both satisfy
$v_2 + 2v_1 = 12 \leq 12$, so both are feasible. By
Lemma~\ref{lem:decomposition}(iv), they correspond to the
configurations $(8,8,8,6)$ and $(8,8,7,7)$, respectively, and only the
former has $v_1-v_4\geq4$. The former is attained by the coloring in
Figure~\ref{fig:ladders}; the latter is attained because $L_{12}$
admits an equitable $4$-total coloring~\cite{cubic}.
\end{proof}
 
\begin{theorem}\label{thm:n16}
In every $4$-total coloring of a cubic graph of order $16$, the color
class configuration is $(10,10,10,10)$ or $(11,10,10,9)$. Both
configurations are attained by a connected cubic graph $H_{16}$; in
particular, $(11,10,10,9)$ is the unique configuration of a
non-equitable coloring.
\end{theorem}

\begin{proof}
Bound (a) of Lemma~\ref{lem:bounds} gives $v_1\leq32/5$, hence
$v_1\leq6$. If $v_1=6$, then bound (b) gives
$v_2\leq(16-6)/2=5$, hence $v_2\leq4$, and the sum condition forces
$(6,4,4,2)$. If $v_1=4$, the sum condition forces $(4,4,4,4)$. If
$v_1\leq2$, the sum is at most $8<16$. Thus the only feasible profiles
are $(6,4,4,2)$ and $(4,4,4,4)$, corresponding to
$(11,10,10,9)$ and $(10,10,10,10)$, respectively.

We refer to Appendix~\ref{app:H16} for an edge list of $H_{16}$ and
for explicit partitions $(S_i,M_i)$ that certify both profiles through
Lemma~\ref{lem:decomposition}. The two colorings are displayed in
Figure~\ref{fig:H16}.
\end{proof}

\begin{figure}[htbp]
\centering
\includegraphics[width=\linewidth]{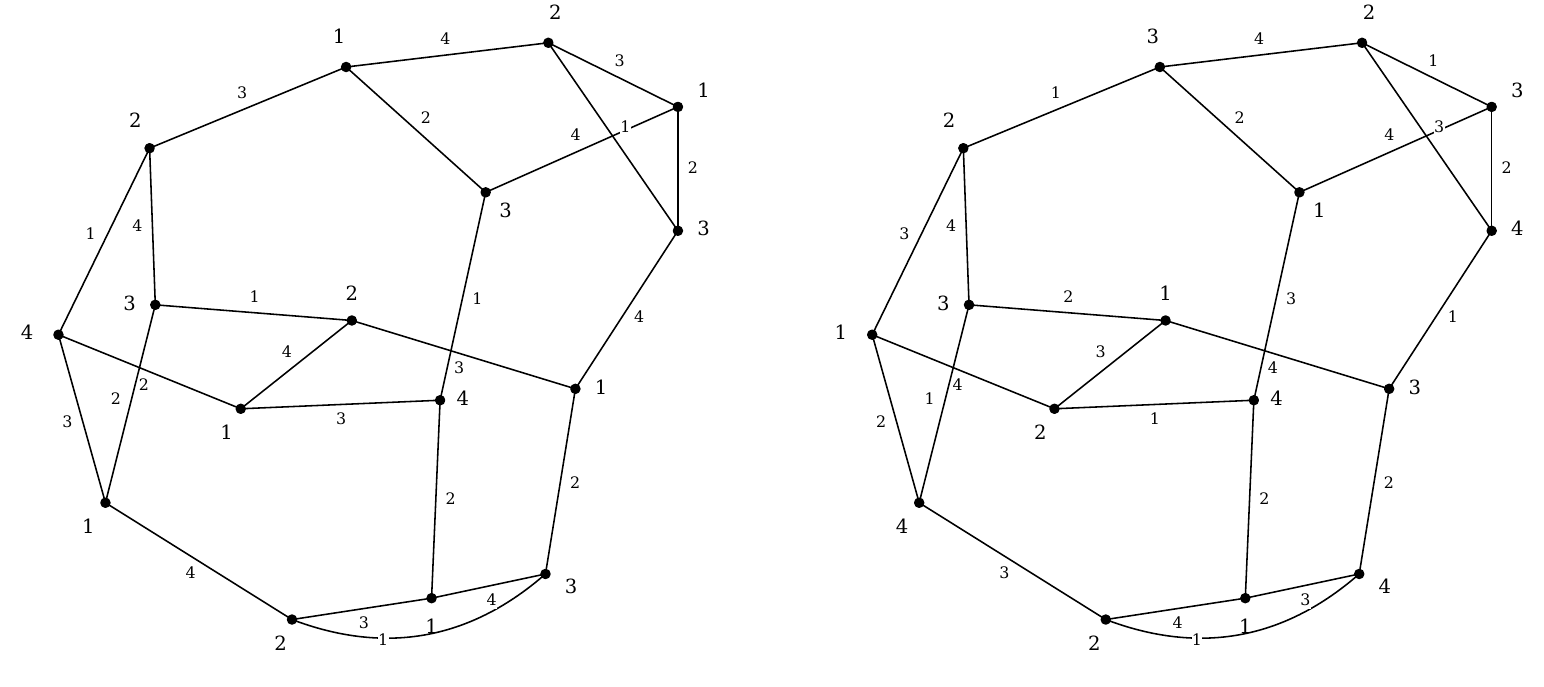}
\caption{Two $4$-total colorings of the same connected cubic graph $H_{16}$. The left coloring has configuration $(11,10,10,9)$ and the right coloring has configuration $(10,10,10,10)$. As in Figure~\ref{fig:ladders}, vertex and edge labels indicate colors.}
\label{fig:H16}
\end{figure}

\begin{theorem}\label{thm:n18}
In every $4$-total coloring of a cubic graph of order $18$, the color
class configuration is one of
\[
(12,11,11,11),\qquad (12,12,11,10),\qquad (12,12,12,9).
\]
All three configurations are attained. Consequently, the
non-equitable configurations are exactly $(12,12,11,10)$ and
$(12,12,12,9)$.
\end{theorem}

\begin{proof}
Bound (a) of Lemma~\ref{lem:bounds} gives $v_1\leq36/5$, hence
$v_1\leq6$. If $v_1\leq4$, the sum of the four entries is at most
$16<18$, so $v_1=6$. The multisets of the form
$(6,v_2,v_3,v_4)$ with even entries at most $6$ summing to $18$ are
$(6,6,6,0)$, $(6,6,4,2)$, and $(6,4,4,4)$; all satisfy
$v_2+2v_1\leq18$. They correspond, respectively, to the configurations
$(12,12,12,9)$, $(12,12,11,10)$, and $(12,11,11,11)$.

The first is attained by $L_{18}$ in Figure~\ref{fig:ladders}. We
refer to Appendix~\ref{app:H18} for the definition of a connected cubic
graph $H_{18}$ and for certificates showing that $H_{18}$ attains each
of the other two configurations; the corresponding colorings appear in
Figure~\ref{fig:H18}.
\end{proof}

\begin{figure}[htbp]
\centering
\includegraphics[width=\linewidth]{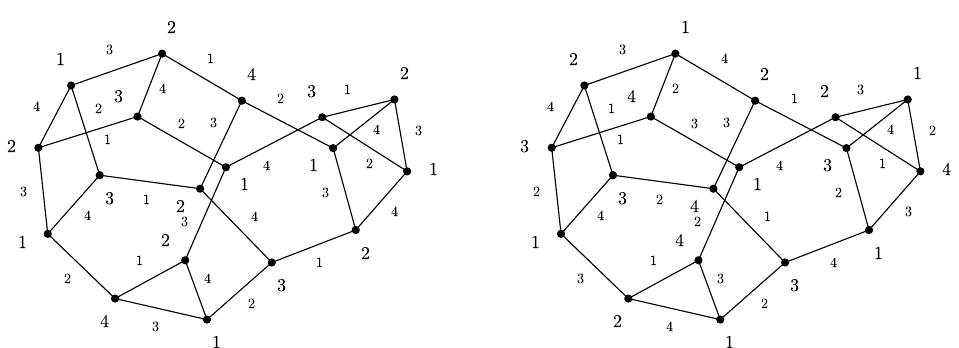}
\caption{Two $4$-total colorings of the same connected cubic graph $H_{18}$. The left coloring has configuration $(12,12,11,10)$ and the right coloring has configuration $(12,11,11,11)$. As in Figure~\ref{fig:ladders}, vertex and edge labels indicate colors.}
\label{fig:H18}
\end{figure}

\begin{remark}\label{rem:otherexamples}
Non-equitable $4$-total colorings of order $12$ and $18$ are not
restricted to circular ladders. Indeed, the generalized Petersen graphs
$G(6,2)$ and $G(9,2)$ admit non-equitable $4$-total colorings with
configurations $(8,8,8,6)$ and $(12,12,12,9)$, respectively.
\end{remark}

\begin{remark}\label{rem:CHladders}
One can check that, in the colorings of circular ladders constructed
in~\cite{chet}, each pasted block contributes three elements to one
fixed color and four elements to each of the other three. The
resulting configurations are $(4k,4k,4k,3k)$ for $L_{6k}$ with
$k\geq1$, $(5{+}4k,5{+}4k,5{+}4k,5{+}3k)$ for $L_{8+6k}$ with
$k\geq0$, and $(10{+}4k,10{+}4k,10{+}4k,10{+}3k)$ for $L_{16+6k}$
with $k\geq0$. Hence the
Chetwynd--Hilton coloring of $L_n$ is non-equitable precisely for the
even orders $n \geq 12$ with $n \notin \{14, 16, 22\}$. Order $14$ is
a genuine exception, by Theorem~\ref{thm:small}; orders $16$ and $22$
are exceptions of this particular construction only, as
Theorem~\ref{thm:n16} (the graph $H_{16}$) and
Theorem~\ref{thm:allorders} below show.
\end{remark}
 
\section{All larger orders}\label{sec:larger}
 
Theorems~\ref{thm:small} and~\ref{thm:n16}, together with the
counterexamples of orders $12$, $16$, $18$, and $20$, raise the natural
question of what happens for orders beyond $20$. In this section we
show that non-equitable $4$-total colorings exist for \emph{every} even
order $n \geq 16$, and so order $14$ is genuinely extremal. The key
tool is a splicing operation that merges two totally colored cubic
graphs.
 
\begin{lemma}[Splice lemma]\label{lem:splice}
Let $G_1$ and $G_2$ be disjoint cubic graphs with $4$-total colorings
$\mathcal{C}_1$ and $\mathcal{C}_2$ over the color set $\{1,2,3,4\}$,
and let $u_1v_1 \in E(G_1)$ and $u_2v_2 \in E(G_2)$. Then there exists
a permutation $\pi$ of $\{1,2,3,4\}$ such that
\[
\pi\bigl(\mathcal{C}_2(u_2v_2)\bigr) = \mathcal{C}_1(u_1v_1), \qquad
\pi\bigl(\mathcal{C}_2(u_2)\bigr) \neq \mathcal{C}_1(u_1), \qquad
\pi\bigl(\mathcal{C}_2(v_2)\bigr) \neq \mathcal{C}_1(v_1).
\]
For every such $\pi$, the graph
\[
G_1 \oplus G_2 \;=\;
\bigl(G_1 \cup G_2\bigr) - \{u_1v_1,\, u_2v_2\} + \{u_1u_2,\, v_1v_2\}
\]
is a simple cubic graph, connected whenever $G_1$ and $G_2$ are
connected and at least one of $u_1v_1$, $u_2v_2$ is not a bridge, and
the assignment $\mathcal{C}$ that coincides with $\mathcal{C}_1$ on the
remaining elements of $G_1$, with $\pi \circ \mathcal{C}_2$ on the
remaining elements of $G_2$, and that assigns the color
$\mathcal{C}_1(u_1v_1)$ to both new
edges $u_1u_2$ and $v_1v_2$, is a $4$-total coloring of
$G_1 \oplus G_2$. Moreover, for each color $i$, the class of color $i$
in $\mathcal{C}$ has cardinality equal to the sum of the cardinalities
of the class of color $i$ in $\mathcal{C}_1$ and of the class of color
$i$ in $\pi \circ \mathcal{C}_2$.
\end{lemma}
 
\begin{proof}
Write $c = \mathcal{C}_1(u_1v_1)$, $d = \mathcal{C}_2(u_2v_2)$,
$a = \mathcal{C}_1(u_1)$, $a' = \mathcal{C}_1(v_1)$,
$b = \mathcal{C}_2(u_2)$, and $b' = \mathcal{C}_2(v_2)$. Since $u_1$
and $v_1$ are adjacent, $a \neq a'$; likewise $b \neq b'$; and
$a, a' \neq c$ while $b, b' \neq d$, by incidence.
 
First we prove the existence of $\pi$. The permutations $\pi$ with $\pi(d) = c$
restrict to bijections from $\{1,2,3,4\} \setminus \{d\}$ onto
$\{1,2,3,4\} \setminus \{c\}$, and there are $3! = 6$ of them. Exactly
$2$ of these violate the constraint $\pi(b) \neq a$, exactly $2$
violate $\pi(b') \neq a'$, and exactly $1$ violates both (this one
exists and is unique because $b \neq b'$ and $a \neq a'$). By
inclusion--exclusion, exactly $6 - 2 - 2 + 1 = 3$ of the six
permutations satisfy both constraints.
 
Next we prove that $G_1 \oplus G_2$ is a simple cubic graph. Deleting
$u_1v_1$ and $u_2v_2$ lowers the degrees of
$u_1, v_1, u_2, v_2$ to $2$, and adding $u_1u_2$ and $v_1v_2$ restores
degree $3$ everywhere, so $G_1 \oplus G_2$ is cubic; it is simple
because the new edges join the two disjoint graphs. For connectivity,
assume that $G_1$ and $G_2$ are connected and,
without loss of generality, that $u_2v_2$ is not a bridge, so that
$G_2 - u_2v_2$ is connected. Every vertex of $G_1 - u_1v_1$ reaches
$u_1$ or $v_1$ within $G_1 - u_1v_1$, since each component of
$G_1 - u_1v_1$ contains $u_1$ or $v_1$; and both $u_1$ and $v_1$ are
joined by the new edges to the connected graph $G_2 - u_2v_2$. Hence
$G_1 \oplus G_2$ is connected. (If both deleted edges are bridges, the
splice may indeed disconnect, which is why the hypothesis is needed.)
 
Finally, we prove that $\mathcal{C}$ is a $4$-total coloring.
Recoloring $G_2$ by the permutation
$\pi \circ \mathcal{C}_2$ preserves the property of being a $4$-total
coloring. Every constraint of $\mathcal{C}$ involving only elements of
$G_1$, or only elements of $G_2$, other than the new edges, is
inherited from $\mathcal{C}_1$ or from $\pi \circ \mathcal{C}_2$. It
remains to check the constraints involving the new edges and the new
adjacency between the two sides. The vertices $u_1$ and $u_2$ are now
adjacent, and their colors are $a$ and $\pi(b) \neq a$; likewise $v_1$
and $v_2$ have colors $a'$ and $\pi(b') \neq a'$. The new edge
$u_1u_2$ has color $c$: at $u_1$, the two remaining edges were adjacent
to $u_1v_1$ in $\mathcal{C}_1$ and hence have colors distinct from
$c$, and the color of $u_1$ is $a \neq c$; at $u_2$, the two remaining
edges were adjacent to $u_2v_2$, hence have colors distinct from $d$ in
$\mathcal{C}_2$ and distinct from $c$ after applying $\pi$, and the
color of $u_2$ is $\pi(b) \neq c$ since $b \neq d$ and $\pi$ is
injective. The same argument
applies to $v_1v_2$. Finally, $u_1u_2$ and $v_1v_2$ share no endpoint,
so they may both receive the color $c$. For the class cardinalities,
observe that, after applying $\pi$ to $\mathcal{C}_2$, exactly two
edges of color $c$ were removed (one from each side) and exactly two
edges of color $c$ were added, while every other element retained its
color. This concludes the proof.
\end{proof}
 
\begin{theorem}\label{thm:allorders}
For every even $n \geq 16$, there exists a connected cubic graph of
order $n$ that admits a non-equitable $4$-total coloring.
\end{theorem}
 
\begin{proof}
We argue by induction on $n$, in steps of $6$, from the three base
cases $n = 16$, $18$, and $20$: the graph $H_{16}$ of
Theorem~\ref{thm:n16} (configuration $(11,10,10,9)$), the circular
ladder $L_{18}$ (configuration $(12,12,12,9)$), and the graph $R$
of~\cite{cubic} (Figure~\ref{fig:R}) admit non-equitable $4$-total
colorings, and all three graphs are connected.
 
For the induction step, let $G$ be a connected cubic graph of order
$n \geq 16$ with a non-equitable $4$-total coloring $\mathcal{C}_1$,
whose class of color $i$ has cardinality $s_i$; since $\mathcal{C}_1$
is non-equitable, $s_{\max} - s_{\min} \geq 2$. Fix a color $\mu$ whose
class has the minimum cardinality $s_{\min}$. By Lemma~\ref{lem:decomposition}, the
class of color $\mu$ contains $(n - v_\mu)/2$ edges, and
$(n - v_\mu)/2 \geq \bigl(n - 2n/5\bigr)/2 > 0$ by
Lemma~\ref{lem:bounds}(a); fix an edge $u_1v_1$ of color $\mu$.
 
Now consider the circular ladder $L_6$, which is Type~1~\cite{chet},
with any $4$-total coloring $\mathcal{C}_2$. By
Theorem~\ref{thm:small} (proof), the vertex profile of $\mathcal{C}_2$
is $(2,2,2,0)$, so its color class configuration is $(4,4,4,3)$ and, by
Lemma~\ref{lem:decomposition}(ii), the class of cardinality $3$
consists of a perfect matching of $L_6$ with three edges; fix an edge
$u_2v_2$ in this class. Since $L_6$ is $2$-edge-connected, the edge
$u_2v_2$ is not a bridge, so Lemma~\ref{lem:splice} applies and
produces a connected graph.
 
Apply Lemma~\ref{lem:splice} to $G$ and $L_6$ with the edges $u_1v_1$
and $u_2v_2$. The permutation $\pi$ maps the color of the class of
cardinality $3$ of $\mathcal{C}_2$ to $\mu$, so in the resulting
$4$-total coloring of $G' = G \oplus L_6$ the class of color $\mu$ has
cardinality $s_{\min} + 3$, while the class of each other color $i$
has cardinality $s_i + 4$. Consequently the minimum class cardinality of this coloring is
$s_{\min}+3$, the maximum is $s_{\max}+4$, and their difference is
$(s_{\max} - s_{\min}) + 1 \geq 3$, so the coloring of $G'$ is
non-equitable. Moreover, $G'$ is a connected cubic graph of order
$n + 6$.
 
Since every even $n \geq 16$ can be written as $n_0 + 6k$ with
$n_0 \in \{16, 18, 20\}$ and $k \geq 0$, this finishes the proof.
\end{proof}
 
Together with Theorem~\ref{thm:small}, Theorem~\ref{thm:allorders}
determines exactly the orders for which the phenomenon behind
Conjecture~\ref{conj:stemock} holds. Notice that every graph produced
above admits a $4$-total coloring, and so is Type~1.
 
\begin{corollary}\label{cor:maximal}
Let $n\geq4$ be an even integer. Every $4$-total coloring of every
cubic graph of order $n$ is equitable if and only if
$n \in \{4,6,8,10,14\}$, where the case $n=4$ holds vacuously
because $K_4$ is Type~2. In particular, $n = 14$ is the largest order
with this property.
\end{corollary}
 
\begin{proof}
For $n \in \{6,8,10,14\}$ this is Theorem~\ref{thm:small}, and for
$n = 4$ the unique cubic graph $K_4$ admits no $4$-total coloring. For
$n = 12$ the ladder $L_{12}$ provides a non-equitable coloring
(Corollary~\ref{cor:smallest}), and for every even $n \geq 16$ a
non-equitable coloring exists by Theorem~\ref{thm:allorders}.
\end{proof}
 
\section{Final remarks}\label{sec:conclusion}
 
The graph $R$ of~\cite{cubic} exhibits an extreme behavior: it is
Type~1 and \emph{all} of its $4$-total colorings are non-equitable,
that is, $\chi''_e(R) = 5 > \chi''(R)$. All the explicit counterexamples of orders $12$, $16$, and $18$
presented here behave differently, since each of $L_{12}$, $L_{18}$,
$H_{16}$, and $H_{18}$ also admits an equitable $4$-total coloring.
Along these lines, we offer the following question, which restricts a
question implicit in~\cite{bryson} to the orders below $20$.
 
\begin{question}\label{q:etcc}
Does every Type~1 cubic graph of order less than $20$ admit at least
one equitable $4$-total coloring? Equivalently, is $R$ a Type~1 cubic
graph with $\chi''_e = 5$ of minimum order?
\end{question}
 
It suffices to consider connected cubic graphs: if each component has
an equitable $4$-total coloring, then the colors can be independently
permuted in the components so that, component by component, the larger
color classes are assigned to currently smallest global color classes;
the resulting $4$-total coloring is equitable. Therefore,
Question~\ref{q:etcc} is finite. Although deciding equitable $4$-total
colorability is NP-complete in general~\cite{cubic}, the question is
within reach of a systematic computational verification over the
catalogue of connected cubic graphs of order at most $18$~\cite{bcgm},
for instance through integer programming or SAT formulations.

Finally, in support of continued pursuit of the subject, we offer the
following problem, which extends the complete analyses of orders $12$,
$16$, and $18$ in Theorems~\ref{thm:n12}--\ref{thm:n18}.

\begin{problem}
Characterize, for each even $n\geq20$ and each feasible profile with
$v_1-v_4\geq4$, the cubic graphs of order $n$ that attain it.
\end{problem}

\appendix
\section{Explicit coloring certificates}\label{app:certificates}
The certificates below were found through integer programming. One can
check that, in every list, the four sets $S_i$ partition the vertex
set, the four sets $M_i$ partition the edge set, each $S_i$ is
independent, and $M_i$ is a perfect matching of the graph minus $S_i$.
Hence, by Lemma~\ref{lem:decomposition}, every list is a stand-alone
certificate of a $4$-total coloring, and its validity does not depend
on the search computation.

\subsection{The graph \texorpdfstring{$H_{16}$}{H16}}\label{app:H16}
Let $V(H_{16})=\{0,1,\ldots,15\}$ and
\[
\begin{aligned}
E(H_{16})=\{&(0,2), (0,8), (0,15), (1,7), (1,9), (1,12), (2,4), (2,13),\\
&(3,5), (3,7), (3,15), (4,10), (4,14), (5,8), (5,11), (6,7),\\
&(6,12), (6,13), (8,15), (9,11), (9,13), (10,12), (10,14), (11,14)\}.
\end{aligned}
\]
A coloring with profile $(6,4,4,2)$ is given by
\begin{itemize}[leftmargin=*]
\item \textbf{Color 1:} $S_1=\{2,\allowbreak 3,\allowbreak 8,\allowbreak 9,\allowbreak 12,\allowbreak 14\}$ and $M_1=\{(0,15),\allowbreak (1,7),\allowbreak (4,10),\allowbreak (5,11),\allowbreak (6,13)\}$.
\item \textbf{Color 2:} $S_2=\{7,\allowbreak 10,\allowbreak 13,\allowbreak 15\}$ and $M_2=\{(0,8),\allowbreak (1,9),\allowbreak (2,4),\allowbreak (3,5),\allowbreak (6,12),\allowbreak (11,14)\}$.
\item \textbf{Color 3:} $S_3=\{0,\allowbreak 4,\allowbreak 5,\allowbreak 6\}$ and $M_3=\{(1,12),\allowbreak (2,13),\allowbreak (3,7),\allowbreak (8,15),\allowbreak (9,11),\allowbreak (10,14)\}$.
\item \textbf{Color 4:} $S_4=\{1,\allowbreak 11\}$ and $M_4=\{(0,2),\allowbreak (3,15),\allowbreak (4,14),\allowbreak (5,8),\allowbreak (6,7),\allowbreak (9,13),\allowbreak (10,12)\}$.
\end{itemize}
A coloring with profile $(4,4,4,4)$ is given by
\begin{itemize}[leftmargin=*]
\item \textbf{Color 1:} $S_1=\{1,\allowbreak 5,\allowbreak 13,\allowbreak 14\}$ and $M_1=\{(0,2),\allowbreak (3,7),\allowbreak (4,10),\allowbreak (6,12),\allowbreak (8,15),\allowbreak (9,11)\}$.
\item \textbf{Color 2:} $S_2=\{7,\allowbreak 9,\allowbreak 10,\allowbreak 15\}$ and $M_2=\{(0,8),\allowbreak (1,12),\allowbreak (2,4),\allowbreak (3,5),\allowbreak (6,13),\allowbreak (11,14)\}$.
\item \textbf{Color 3:} $S_3=\{2,\allowbreak 3,\allowbreak 6,\allowbreak 8\}$ and $M_3=\{(0,15),\allowbreak (1,7),\allowbreak (4,14),\allowbreak (5,11),\allowbreak (9,13),\allowbreak (10,12)\}$.
\item \textbf{Color 4:} $S_4=\{0,\allowbreak 4,\allowbreak 11,\allowbreak 12\}$ and $M_4=\{(1,9),\allowbreak (2,13),\allowbreak (3,15),\allowbreak (5,8),\allowbreak (6,7),\allowbreak (10,14)\}$.
\end{itemize}

\subsection{The graph \texorpdfstring{$H_{18}$}{H18}}\label{app:H18}
Let $V(H_{18})=\{0,1,\ldots,17\}$ and
\[
\begin{aligned}
E(H_{18})=\{&(0,5), (0,8), (0,17), (1,4), (1,12), (1,15), (2,5), (2,8),\\
&(2,14), (3,6), (3,9), (3,17), (4,6), (4,7), (5,11), (6,12),\\
&(7,13), (7,15), (8,17), (9,13), (9,16), (10,11), (10,14), (10,15),\\
&(11,12), (13,16), (14,16)\}.
\end{aligned}
\]
A coloring with profile $(6,6,4,2)$ is given by
\begin{itemize}[leftmargin=*]
\item \textbf{Color 1:} $S_1=\{1,\allowbreak 3,\allowbreak 5,\allowbreak 7,\allowbreak 8,\allowbreak 16\}$ and $M_1=\{(0,17),\allowbreak (2,14),\allowbreak (4,6),\allowbreak (9,13),\allowbreak (10,15),\allowbreak (11,12)\}$.
\item \textbf{Color 2:} $S_2=\{0,\allowbreak 2,\allowbreak 4,\allowbreak 9,\allowbreak 10,\allowbreak 12\}$ and $M_2=\{(1,15),\allowbreak (3,6),\allowbreak (5,11),\allowbreak (7,13),\allowbreak (8,17),\allowbreak (14,16)\}$.
\item \textbf{Color 3:} $S_3=\{6,\allowbreak 14,\allowbreak 15,\allowbreak 17\}$ and $M_3=\{(0,8),\allowbreak (1,12),\allowbreak (2,5),\allowbreak (3,9),\allowbreak (4,7),\allowbreak (10,11),\allowbreak (13,16)\}$.
\item \textbf{Color 4:} $S_4=\{11,\allowbreak 13\}$ and $M_4=\{(0,5),\allowbreak (1,4),\allowbreak (2,8),\allowbreak (3,17),\allowbreak (6,12),\allowbreak (7,15),\allowbreak (9,16),\allowbreak (10,14)\}$.
\end{itemize}
A coloring with profile $(6,4,4,4)$ is given by
\begin{itemize}[leftmargin=*]
\item \textbf{Color 1:} $S_1=\{0,\allowbreak 2,\allowbreak 3,\allowbreak 7,\allowbreak 12,\allowbreak 16\}$ and $M_1=\{(1,15),\allowbreak (4,6),\allowbreak (5,11),\allowbreak (8,17),\allowbreak (9,13),\allowbreak (10,14)\}$.
\item \textbf{Color 2:} $S_2=\{1,\allowbreak 11,\allowbreak 13,\allowbreak 17\}$ and $M_2=\{(0,8),\allowbreak (2,5),\allowbreak (3,9),\allowbreak (4,7),\allowbreak (6,12),\allowbreak (10,15),\allowbreak (14,16)\}$.
\item \textbf{Color 3:} $S_3=\{4,\allowbreak 5,\allowbreak 14,\allowbreak 15\}$ and $M_3=\{(0,17),\allowbreak (1,12),\allowbreak (2,8),\allowbreak (3,6),\allowbreak (7,13),\allowbreak (9,16),\allowbreak (10,11)\}$.
\item \textbf{Color 4:} $S_4=\{6,\allowbreak 8,\allowbreak 9,\allowbreak 10\}$ and $M_4=\{(0,5),\allowbreak (1,4),\allowbreak (2,14),\allowbreak (3,17),\allowbreak (7,15),\allowbreak (11,12),\allowbreak (13,16)\}$.
\end{itemize}

\bibliographystyle{plain}
\bibliography{references}

\end{document}